\documentclass[12pt, a4paper]{amsart}

\usepackage{amsmath}
\usepackage{amssymb}
\usepackage{amscd}
\usepackage{latexsym}
\usepackage{amsthm}
\usepackage[all]{xy}

\input xypic

\newtheorem{prop}[equation]{Proposition}

\newtheorem{thm}[equation]{Theorem}
\newtheorem{cor}[equation]{Corollary}
\newtheorem{lem}[equation]{Lemma}

\theoremstyle{definition}

\newtheorem{rem}[equation]{Remark}
\newtheorem{exa}[equation]{Example}
\newtheorem{exas}[equation]{Examples}

\numberwithin{equation}{section}

\newcommand{\Hom}{\operatorname{Hom}}

\newcommand{\letbe}{\mathbin{\raisebox{.3pt}{:}\!=}}

\newcommand{\hocolim}{\operatorname{hocolim}}
\newcommand{\colim}{\operatorname{colim}}
\newcommand{\holim}{\operatorname{holim}}

\newcommand{\bA}{\mathbb{A}}

\newcommand{\bC}{\mathbb{C}}

\newcommand{\bQ}{\mathbb{Q}}
\newcommand{\bR}{\mathbb{R}}
\newcommand{\bZ}{\mathbb{Z}}

\newcommand{\srf}[1]{\mbox{${\text{\it SR}^F}$}}

\newcommand{\djs}{\mbox{\it DJ\/}}

\newcommand{\cat}[1]{\mbox{\sc #1}}

\newcommand\Z{\bZ}
\newcommand\z{\Z}
\newcommand\q{\bQ}
\newcommand\Q{\bQ}
\newcommand\C{\bC}
\newcommand\R{\bR}

\newcommand \lra{\longrightarrow}
\newcommand \p{^\wedge_p}
\newcommand \com{^\wedge}
\newcommand \ra{\rightarrow}
\newcommand \da{\downarrow}
\def \larrow#1{\,\stackrel{#1}\lra\,}

\def \link{\operatorname{\ell}}

\newcommand \af{{\bA_f}}

\newcommand{\const}{\operatorname{cst}}

\newcommand{\catk}{\cat{cat}(K)}

\newcommand{\catkop}{\cat{cat}(K)^{op}}

\newcommand{\ab}{\cat{ab}}

\newcommand{\Top}{\cat{Top}}
\newcommand{\nz}{\newline}
\newcommand{\quer}{\overline}

\begin{document}
\bibliographystyle{plain}

\title{Vector
bundles over
Davis-Januszkiewicz  spaces with prescribed characteristic classes}

\author{Dietrich Notbohm}
\address{
Department of Mathematics,
Vrije Universiteit,
Faculty of Sciences,
De Boolelaan 1081a,
1081 HV Amsterdam,
The Netherlands
}
\email{notbohm@few.vu.nl}



\begin{abstract}

For any $(n-1)$-dimensional simplicial complex, we construct a particular
$n$-dimensional complex vector bundle over
the associated Davis-Januszkiewicz
space whose Chern classes are given by the elementary symmetric polynomials
in the generators of the Stanley Reisner algebra. We show that the isomorphism type
of this complex vector bundle as well as of its realification are
completely determined by its characteristic classes. This allows us to show that coloring properties of the
simplicial complex are reflected by splitting properties of this bundle and vice versa.
Similar question are also discussed for
$2n$-dimensional real vector bundles with particular prescribed characteristic Pontrjagin and Euler classes.
We also analyze which of these bundles admit a complex structure.
It turns out that all these bundles are closely related to the tangent bundles of
quasitoric manifolds and moment angle complexes.
\end{abstract}

\maketitle

\markright{VECTOR BUNDLES} 

\section{Introduction}\label{introduction}

For any finite  simplicial complex $K$, Davis and Januszkiewicz constructed a family
of realizations of the Stanley-Reisner algebra $\Z[K]$, that is the integral cohomology of these
spaces is isomorphic to $\Z[K]$. They also showed that all these spaces are homotopy equivalent
\cite[Section 4]{daja}. We denote a generic
model for this homotopy type  by $\djs(K)$. In Section 6 of the mentioned paper, Davis and Januszkiewicz also
studied particular vector bundles over DJ(K).  They
constructed an $m$-dimensional complex bundle $\lambda\da DJ(K)$ whose Chern classes
are given by the elementary symmetric polynomials in the generators of the Stanley Reisner algebra
$\Z[K]$ and compared it with vector bundles obtained from the tangent bundles of
the moment angle complex or a quasitoric manifold by an application of the Borel construction.
In particular, they showed that these bundles are stably isomorphic and have the same Pontrjagin classes.
We were wondering to which extend the characteristic classes determine the isomorphism type of these bundles.

We will split off a large trivial vector bundle from $\lambda$ and will show that
the isomorphism type of the remaining vector bundle $\xi\da \djs(K)$ as well as of its
realification $\xi_\R$ is completely determined
by its characteristic classes. We will also study real vector bundles $\rho$ with the same Pontrjagin classes
as $\xi$ and prove for them similar existence and uniqueness results in terms Pontrjagin and Euler classes.
We are able to offer several applications of the uniqueness results:
Colorings of $K$ are reflected by  (stable) splittings of $\xi$ and $\xi_\R$ into a direct sum
of complex line bundles or of $2$-dimensional real bundles. We will also improve on the stable
isomorphisms results mentioned above and classify complex structures  on the real vector bundles $\rho$.

To make our statements more precise we need to fix notation and
to recall some basic constructions.
An \emph{abstract simplicial complex} on the finite vertex set $V$ of order
$|V|=m$ is a
set $K=\{\alpha_1,....,\alpha_s\}$ of subsets $\alpha_i\subset V$
which is closed under the formation of subsets and which contains the empty set $\emptyset$.
We will always identify $V$ with the set $[m]\letbe\{1,...,m\}$ of the first $m$ natural numbers.
The \emph{dimension} $\dim K$ of $K$ is expressed in terms of the cardinality of its faces $\alpha\in K$.
We set $\dim \alpha\letbe |\alpha|-1$ and $\dim K$ is the maximum of the dimensions of its faces.
Some of our statements involve sums taken over all maximal faces. We will denote this set by $M_K$.
Examples are given by full simplices and their boundaries.
For a set $\alpha$ we denote by $\Delta[\alpha]$ the simplicial complex given by
all subsets of $\alpha$ and by
$\partial\Delta[\alpha]$ the complex of all proper subsets of $\alpha$.
Then $\dim \Delta[\alpha]=|\alpha|-1$ and $\dim \partial\Delta[\alpha]=|\alpha|-2$.
The set $M_{\Delta[\alpha]}$
consist only of the set $\alpha$ and $M_{\partial\Delta[\alpha]}$
consists of all subsets of $\alpha$ of order $|\alpha|-1$.
The complex $K$ is a subcomplex of the full simplex $\Delta[m]$ on $m$-vertices and for each $\alpha\in K$,
the complex $\Delta[\alpha]\subset K$ is a subcomplex of $K$.

For a commutative ring $R$ with unit we denote by $R[m]\letbe R[v_1,...,v_m]$
the
graded polynomial algebra generated by algebraically independent
elements $v_1,...,v_m$ of degree 2, one for each vertex of $K$.
For each subset $\alpha\subseteq [m]$ we denote by
$v_\alpha\letbe \prod_{j\in \alpha} v_j$ the square free monomial whose factors
are in one to one correspondence with vertices contained in $\alpha$.
The \emph{graded
Stanley-Reisner algebra} $R[K]$ associated with $K$
is defined as the quotient $R[K]\letbe R[m]/I_K$,
where  $I_K\subset R[m]$ is the ideal generated
by all elements $v_\beta$ such that $\beta\subseteq [m]$ is not a face of $K$.

For a complex vector bundle $\eta$ over a space $X$ we denote by $c_i(\eta)\in H^{2i}(X;\Z)$
the $i$-th Chern class and by
$c(\eta)\letbe 1+\Sigma_{i\geq 1} c_i(\eta)$ the total Chern class of $\eta$.
For a real vector bundle $\rho$ over $X$ we use
$p_i(\rho)\in H^{4i}(X;\Z)$  and
$p(\rho)\letbe 1 + \Sigma_{i\geq 1} p_i(\rho)$ to denote the $i$-th  and
the total Pontrjagin class of $\rho$.
If $\rho$ is $k$-dimensional and oriented, the Euler class is denoted by $e(\rho)\in H^k(X;\Z)$.

We fix an isomorphism $H^*(\djs(K);\z)\cong \z[K]$. Since $BT^m$ is an Eilenberg-MacLane
space realizing the algebra $\Z[m]$, there exist a map $q_K\colon \djs(K)\lra BT^m$
which induces in cohomology
the projection $\z[m]\lra \z[K]$.  We can think of $T^m$ as the maximal torus $T^m\subset U(m)$.
The pull back of the $m$-dimensional universal complex vector bundle over $BU(m)$ along
the composition of $q_K$ with the maximal torus inclusion $BT^m \lra BU(m)$
produces a vector bundle $\lambda\da \djs(K)$ whose total Chern class is
$c(\lambda)= \prod_{i=1}^m (1+v_i)$. And the total Pontrjagin class of the realification $\lambda_\R$ is
$p(\lambda_\R)=\prod_{i=1}^m (1-v_i^2)$ \cite[Section 6]{daja}. These characteristic classes play a
particular role
in our statements and we define $c(K)\letbe\prod_{i=1}^m(1+v_i)$ and
$p(K)\letbe \prod_{i=1}^m (1-v_i^2)$.

Confusing notation we also denote by $\C$ and $\R$
 1-dimensional trivial
complex or real  vector  bundles over a space $X$.

\begin{thm} \label{thm1}
Let $K$ be an $(n-1)$-dimensional  finite simplicial complex.
Then there exists an $n$-dimensional
complex vector bundle $\xi$ over $\djs(K)$ such that
$\lambda\cong \xi\oplus \C^{m-n}$. In particular,
$c(\xi)=c(K)$ and
$p(\xi_\R)=p(K)$.
Moreover, $c_n(\xi)$ and $p_n(\xi_\R)$ are non trivial.
\end{thm}

The vector bundle $\xi\da DJ(K)$ also  satisfies uniqueness properties. Its isomorphism type
is completely determined by its characteristic classes. The realification $\xi_{\R}$
of a complex bundle $\xi$ carries a
canonical orientation. And $\xi_\R$ denotes the underlying oriented real vector bundle with
this canonical orientation.
The Euler class $e(\xi_{\R})$ is then given by the $n$-the Chern class
$c_n(\xi)$. Trivial real vector bundles get the standard orientation.

\begin{thm} \label{thm2}
Let $K$ be an $(n-1)$-dimensional  finite simplicial complex. Let $\eta\da DJ(K)$ be
an $r$-dimensional complex and $\rho\da DJ(K)$ an oriented $s$-dimensional real vector bundle
over $DJ(K)$.
\nz
(i) If $c(\eta)=c(\xi)$, then $r\geq n$ and $\eta$ and $\xi\oplus \C^{r-n}$
are isomorphic.
\nz
(ii) If $p(\rho)=p(\xi_{\R})$, then $s\geq 2n$. If $s>2n$, then
$\rho$ and $\xi_{\R}\oplus \R^{s-2n}$ are isomorphic.
And if $s=2n$,
then $\rho\oplus \R$ and $\xi_{\R}\oplus \R$ are isomorphic.
\nz
(iii) If $s=n$, $p(\rho)=p(\xi_{\R})$ and $e(\rho)=e(\xi_\R)$, then $\rho$ and $\xi_\R$ are isomorphic.
\end{thm}

For clarification, the isomorphisms in Part (ii) and (iii) are isomorphisms of oriented bundles.

In Section 4 we will allow more general Euler classes. A function $\omega\colon M_K\lra \{\pm 1\}$
gives rise to class $e_\omega(K)\letbe\sum_{\mu\in M_K} \omega(\mu) v_\mu \in H^{2n}(DJ(K);\Z)$.
For each function $\omega$ we will construct an oriented real vector bundle $\rho_\omega$ over $DJ(K)$
with $e(\rho_\omega)=e_\omega(K)$ and $p(\rho_\omega)=p(K)$ (Theorem \ref{globalexistence}).
The bundles $\rho_\omega$ also satisfy uniqueness results in terms of Euler and
Pontrjagin classes  (Theorem \ref{globaluniqueness}) and
provide a complete list of isomorphism types of  (oriented) $2n$-dimensional real vector bundles
whose total Pontrjagin class equals $\prod_{i\in [m]}(1-v_i^2)$

\begin{thm} \label{thm3}
Let $K$ be an $n-1$-dimensional finite simplicial complex.
Let $\rho\da DJ(K)$
be a oriented $2n$-dimensional real vector bundle
such that $p(\rho)= \prod_{i\in [m]}(1-v_i^2)$. Then there exists a function
$\omega\colon M_K\lra \{\pm 1\}$ such that $\rho$ and $\rho_\omega$ are isomorphic.
\end{thm}

The realification $\xi_\R$ satisfies the assumption of the last theorem,
realizes the function $\omega \equiv 1$, i.e. $\omega(\mu)=1$ for all $\mu\in M_K$,
and admits obviously a complex structure.
We say that a real vector bundle $\rho$ {\it admits a
complex structure} if there exists a complex bundle $\eta$ such that $\rho$ and $\eta_\R$ are isomorphic
as non-oriented real vector bundles. For a function $f\colon [m]\lra \{\pm 1\}$ we define a
function $\omega_f\colon M_K\lra \{\pm 1\}$ by $\omega_f(\mu)\letbe \prod_{i\in \mu} f(i)$
and denote $e_{\omega_f}(K)$ by $e_f(K)$.
.

\begin{thm} \label{thm4}
Let $K$ be an $n-1$-dimensional finite simplicial complex. Let $\rho$ be an $2n$-dimensional
oriented real vector bundle over $DJ(K)$ such that $p(\rho)=p(K)$.
Then, the bundle $\rho$ admits a complex structure if and only if
there exists a function $f\colon [m]\lra \{\pm 1\}$ such that $e(\rho)=\pm e_f(K)$.
\end{thm}

\begin{rem}\label{nonoriented}
Since Davis-Januszkiewicz spaces are simply connected, every real vector bundle $\rho$ over
$DJ(K)$ is orientable. The Euler
classes for different orientations may only differ by a sign.
In Theorem \ref{thm2}, Theorem \ref{thm3} and Theorem \ref{thm4} we made the assumption that $\rho$
is oriented. This is not a serious restriction. For example, if $\eta$ is a complex
and $\rho$ a real vector bundle over $DJ(K)$ such that $\rho$ and $\eta_\R$ are isomorphic
as non-oriented vector bundles,
then the canonical orientation on $\eta_\R$ must agree with one of the possible orientation
for $\rho$. That is that $\xi_\R$ and $\rho$ with the appropriate orientation
are isomorphic as oriented bundles.
\end{rem}

Several of the vector bundles will be constructed
as homotopy orbit spaces.
For a compact Lie group $G$ and a $G$-space $X$,
the {\it Borel construction} or {\it homotopy orbit space}
$EG\times_G X$ will be denoted by $X_{hG}$. If $\eta$ is an $n$-dimensional
$G$-vector bundle over $X$
with total space $E(\eta)$, the Borel construction establishes a fibre bundle
$E(\eta)_{hG} \lra X_{hG}$. In fact, this is an $n$-dimensional vector bundle over
$X_{hG}$ \cite{segal}, denoted by $\eta_{hG}$.

There are several connections with the work of Davis and Januszkiewicz, in particular with
Section 6 of
\cite{daja}.
For every simplicial complex $K$ there exists an associated moment angle complex
$Z_K$ which carries a $T^m$-action.
If $K$ is the triangulation of an $(n-1)$-dimensional sphere, the moment angle complex
$Z_K$
is an $(m+n)$-dimensional manifold.
The tangent bundle $\tau_Z$ is a $T^m$-equivariant
real vector bundle. Applying the Borel construction establishes
a vector bundle $(\tau_Z)_{hT^m}$ over $(Z_K)_{hT^m}\simeq DJ(K)$ which is stably isomorphic to
$\lambda_\R$.
In particular, $(\tau_Z)_{hT^m}$ and $\xi_\R$ have equal Pontrjagin classes.
Some details of the construction are recalled in the next section.

An oriented quasi toric manifold $M^{2n}$  is an oriented $2n$-dimensional manifold
with an $T^n$-action such that the action is locally standard and
the orbit space $M^{2n}/T^n=P$ is a
simple polytope. We apply the Borel construction to the tangent bundle $\tau_M$ of $M^{2n}$
and get a vector bundle $(\tau_m)_{hT^n}$ over $(M^{2n})_{hT^n}\simeq \djs(K_P)$, where
$K_P$ is the simplicial complex dual to the boundary of $P$. Again,
$(\tau_M)_{hT^n}$ and $\lambda_\R$ are stably isomorphic and $(\tau_M)_{hT^n}$ and $\xi_\R$ have the same Pontrjagin classes.
All this can be found in \cite[Sections 4 and 6]{daja}.

Theorem \ref{thm2} allows to draw the following corollary, which improves the stable isomorphism result about
$\lambda$ and $(\tau_{M})_{hT^n}$
\cite[Theorem 6.6]{daja}

\begin{cor} $\ \ \ $
\nz
(i) If $K$ is the triangulation of an $(n-1)$-dimensional sphere,
then $(\tau_Z)_{hT^m}$ and $\xi_\R\oplus \R^{m-n}$ are isomorphic.
\nz
(ii) If $M^{2n} \lra P$ is an oriented
quasi toric manifold over the simple polytope $P$,
then $(\tau_M)_{hT^n}\oplus \R$ and $\xi_\R\oplus \R$ are isomorphic.
\end{cor}

The connected sum $\C P(2)\sharp \C P(2)$ of $\C P(2)$ with itself is a quasitoric manifold
over a square, but does not admit an almost complex structure
\cite[Section 0 (C') and Example 1.19]{daja} \cite[page 68]{bupa}. In particular, for $M=\C P(2)\sharp \C P(2)$,
the real vector  bundle
$(\tau_M)_{hT^n}$ has no complex structure.
This shows that in this case $\xi_\R$ and $(\tau_M)_{hT^n}$ cannot be  isomorphic and that we cannot
omit the trivial summand in the second part of the above Corollary as well as in Theorem \ref{thm2}(ii).

Equivariant and non-equivariant almost complex structures for
quasi toric manifolds are discussed in \cite[Section 5]{bupa}.
Every equivariant complex structure for a quasi toric manifold $M$
establishes a complex structure for
the Borel construction $ (\tau_M)_{hT^n}$.
Since the total Pontrjagin class of $(\tau_M)_{hT^n}$ equals $\prod_{i=1}^m(1-v_i^2)$,
Theorem \ref{thm4} tells us that the Euler class $e((\tau_M)_{hT^n})$ decides whether
$(\tau_M)_{hT^n}$ admits a complex structure.
The calculation of $e((\tau_M)_{hT^n})$ can be worked out, but will be discussed
in a forthcoming paper \cite{dono}. It turns out that $(\tau_M)_{hT^n}$ has a complex structure if and only
$\tau_M$ has one. The combinatorial conditions for both cases are the same.
For a few more details see Section 7.

In a recent paper \cite{kustarev}, Kustarev gave a combinatorial condition, described in terms of
omniorientations of the quasi toric manifold $M$,
which is sufficient and necessary for the existence of an equivariant
almost complex structure for M.
One can show that his condition and our condition established
in Theorem \ref{thm4} are equivalent\cite{dono}. In contrast to his result we are able
to give a complete classification of all complex structures for $(\tau_M)_{hT^n}$.
For example, we have the following classification result. We say that two complex structure $\eta_1$ and $\eta_2$ for an real
vector vector bundle $\rho$ are isomorphic if $\eta_1\cong \eta_2$ as complex vector bundles.

\begin{cor} \label{cor7}
Let $K$ be an $(n-1)$-dimensional abstract finite simplicial complex and let $s\geq n$. Let $\rho$ be an
$2s$-dimensional
oriented real vector bundle over $DJ(K)$ such that $p(\rho)=p(K)$.
\nz
(i) If $s=n$, the number of isomorphism classes of
complex structures for $\rho$ equals the number of functions $f\colon [m]\lra \{\pm1\}$ such that
$e_f(K)=\pm (\rho)$. In particular, there exists only a
finite number of isomorphism classes of complex structures for $\rho$.
\nz
(ii) If $s>n$ then there exists exactly $2^m$ non isomorphic complex structures for $\rho$.
\end{cor}

In Section \ref{cxstructure} we will give an explicit description of the complex structures of the bundle $\rho$.
In fact, for each function $f\colon [m]\lra \{\pm 1\}$, we will construct an $n$-dimensional complex vector bundle
$\xi_f$ over $DJ(K)$. The complex structures are represented by the bundles $\xi_f$ with
$e_f(K)=\pm e(\rho)$ respectively $\xi_f\oplus \C^{s-n}$.

Theorem \ref{thm2} also allows to relate colorings of simplicial complexes to
stable splittings of the vector bundle
$\xi$.
A regular  {\it $r$-paint coloring}, an {\it $r$-coloring} for short,
of a simplicial complex $K$ is a non degenerate simplicial map
$K\lra \Delta[r]$, i.e. each face of $K$ is mapped isomorphically onto a
face of $\Delta[r]$.
If $\dim K=n-1$, then $K$ only allows $r$-colorings for $r\geq n$.
The inclusion $K\subset \Delta([m])$ always provides an m-coloring.

\begin{cor}
For an $(n-1)$-dimensional simplicial complex $K$ the following conditions are equivalent
\nz
(i) $K$  admits an $r$-coloring $K \lra \Delta[r]$.
\nz
(ii)
The vector bundle $\xi \oplus \C^{r-n}\cong \bigoplus \nu_i$ splits into a direct sum of
complex line bundles.
\nz
(iii)
The vector bundle $\xi_\R\oplus \R^{2(r-n)}\cong \bigoplus_{i=1}^r \theta_i$ splits into a direct
sum of 2-dimensional real vector bundles.
\end{cor}

\begin{proof}
The proof is based on a similar result for colorings and splittings of $\lambda$ and $\lambda_\R$.
We can assume that $n\leq r\leq m$.
An $r$-coloring $K\lra \Delta([r])$ establishes a splitting
$\lambda\cong \bigoplus_{i=1}^r \nu_i\oplus \C^{m-r}$ \cite{no2}.
The bundles $\bigoplus_{i=1}^r \nu_i$ and $\xi\oplus\C^{r-n}$ have the same Chern classes.
Theorem \ref{thm2} shows that both vector bundles are isomorphic and produces the
desired splitting for $\xi\oplus\C^{r-n}$. The realification of both sides
establishes a splitting for $\xi_\R\oplus \R^{2(m-r)}$.
A splitting of $\xi_\R\oplus \R^{2(r-n)}$ produces a splitting for
$\lambda_\R\cong \xi_\R \oplus \R^{2(m-n)}$. And such a splitting establishes a coloring
$K\lra \Delta([r])$ \cite{no2}.
\end{proof}

The paper is organized as follows. In the next section we describe different
models for $DJ(K)$ needed for our purposes. A geometric
construction of the vector bundle $\xi$ with the properties stated in
Theorem \ref{thm1} is  contained in Section \ref{xi}. Using
Sullivan's arithmetic square
we reduce the global uniqueness problem for our bundles to the analogous $p$-adic question
in
Section \ref{global uniqueness}. This section also contains a proof of Theorem \ref{thm3}.
The $p$-adic case
involves the calculation of some higher derived inverse limits.
In Section 5 we provide some general remarks about higher limits and filtrations of functors,
needed in Section \ref{p-adic case} for the proof of the $p$-adic existence and uniqueness statements.
In the last section we discuss complex structures for real vector bundles and prove Theorem \ref{thm4}
and Corollary \ref{cor7}.

If not specified otherwise, $K$ will always denote an abstract
finite simplicial complex of dimension $n-1$ with $m$ vertices and  $M_K$ the set of maximal faces,
We also fix the classes $c(K)=\prod_{i=1}^m(1+v_i)$ and $p_K=\prod_{i=1}^m(1-v_i^2)$.

We would like to express our gratitude to Natalia Dobrinskaya, Taras Panov and Nigel Ray for many
helpful discussions.

\section{Models for $DJ(K)$} \label{models}

For the proof of our theorems we will need different models for
the space $DJ(K)$,
which we will describe in this section.

Let $\catk$ be the category whose objects are the faces of $K$ and whose arrows are given by
the subset relations between the faces. $\catk$ has an initial object given by the empty face.
Given a pair $(X,Y)$ of pointed topological space we can define  covariant functors
$$
X^K,\ (X,Y)^K    
$$
The functor $X^K$ assigns to each face $\alpha$ the cartesian product
$X^\alpha$ and to each morphism $i_{\alpha,\beta}$ the inclusion
$X^\alpha \subset X^\beta$ where
missing coordinates are set to the base point $*$. If $\alpha=\emptyset$, then $X^\alpha$ is a point.
And $(X,Y)^K$ assigns to
$\alpha$ the product
$X^\alpha\times Y^{[m]\setminus \alpha}$ and to $i_{\alpha,\beta}$
the coordinate
wise inclusion
$X^\alpha\times Y^{[m]\setminus \alpha}\subset
X^\beta\times Y^{[m]\setminus \beta}$.
The inclusions $X^\alpha \subset X^{[m]}=X^m$ and
$X^\alpha\times Y^{[m]-\alpha}\subset X^\alpha\times X^{[m]-\alpha}=X^m$
establish
inclusions
$$\colim_{\catk}X^K \lra X^m,\ \
\colim_{\catk}(X,Y)^K \lra X^m.
$$

We are interested in two particular cases, namely the functor $X^K$ for the
classifying space $BT=\C P(\infty)$ of the 1-dimensional
circle $T$ (considered as a topological group) and the functor
$(X,Y)^K$ for the pair $(D^2,S^1)$.
The colimit
$$
Z_K\letbe \colim_{\catk}\, (D^2,S^1)^K
$$
is called the {\it moment angle
complex} associated to $K$. We have inclusions $Z_K \subset (D^2)^m\subset \C^m$
and the standard $T^m$-action on $\C^m$ restricts to $Z_K$.
The Borel construction produces a fibration
$$
q_K\colon (Z_K)_{hT^m} \lra BT^m
$$
with fiber $Z_K$. Moreover, $B_T K\letbe (Z_K)_{hT^m}$ is a realization of
the Stanley-Reisner algebra $\Z[K]$. That is there exists an isomorphism
$H^*(B_TK;\Z)\cong \Z[K]$  such that the map $H^*(q_K;\bZ)$
can be identified with the map $\bZ[m]\lra \Z[K]$ \cite[Theorem 4.8]{daja}.
We will use this model for the geometric construction of
our vector bundle $\xi$.

Buchstaber and Panov gave a different construction for $DJ(K)$.
They showed that the colimit
$c(K)\letbe \colim_{\catk}\, BT^K$ is
homotopy equivalent to $B_TK$ and that
the inclusion
$$
c(K) \lra BT^m
$$
is homotopic to $q_K$ \cite[Theorem 6.29]{bupa}.

We wish to study homotopy theoretic properties of $c(K)$.
Colimits behave poorly from a homotopy theoretic point of view,
but the left derived functor, known as homotopy colimit,
provides the appropriate tool for such questions.
Following \cite{vogt}, the homotopy colimit
$hc(K)\letbe \hocolim_{\catk}\, BT^K$
may be described as the two sided bar
construction $B(*,\catk,BT^K)$ in $\Top$. For the functor $BT^K$,
the map
$$
\hocolim_{\catk}\, BT^K \lra \colim_{\catk}\, BT^K
$$
is a homotopy equivalence \cite{nora1}. In particular,
$hc(K)\simeq c(K)$.
We will use the model $hc(K)$ to prove the uniqueness properties
of the vector bundle $\xi$ and to provide a homotopy
theoretic construction of $\xi$.

For later purpose we include the following remark.

\begin{rem} \label{limitzk}
For $\alpha\subset [m]$ we denote by $\Z[\alpha]$ the polynomial algebra generated by elements $v_i$, $i\in \alpha$, of degree 2,
one for each vertex in $\alpha$. The projection $\Z[m]\lra \Z[\alpha]$ which maps $v_i$ to $v_i$ if $i\in \alpha$
and to 0 otherwise is induced by the inclusion $i_{\alpha,[m]}\colon BT^\alpha\lra BT^m$.
By construction, this projection factors uniquely through a map
$h^\alpha\colon \Z[K] \lra \Z[\alpha]$. And this map is induced by the inclusion
$BT^\alpha\lra DJ(K)$ which comes from the description of $DJ(K)$
as colimit or homotopy colimit. The direct sum of the maps $h^\alpha$
induces a monomorphism of algebras $\Z[K]\lra \bigoplus_{\alpha\in K} \Z[\alpha]$.
\end{rem}

\section{Geometric construction of vector bundles} \label{xi}

The $m$-dimensional torus $T^m$ acts on $\C^m$ via coordinate wise multiplication.
Let $\lambda'\da Z_K$ denote the $m$-dimensional
$T^m$-equivariant complex vector bundle over $Z_K$ given by the diagonal action of
$T^m$ on $\C^m\times Z_K$. The Borel construction produces an $m$-dimensional complex bundle
$\lambda\letbe \lambda'_{hT^m}$ over $DJ(K)$ which is isomorphic to the pull back of the universal bundle
over $BU(m)$ along the composition of $q_K\colon DJ(K)\lra BT^m$ and the maximal torus inclusion $BT^m \lra BU(m)$.
As already mentioned, we have $c(\lambda)=c(K)$ (for details see \cite[Section 6]{daja}).

For $x,z\in \C^m$ we denote by $\quer z$ the coordinate wise complex conjugate of $z$ and
by $xz\in \C^m$ the coordinate wise product of $x$ and $z$; that is that the $i$-th coordinate
$(xz)_i$ of $x z$ is given by the product $x_iz_i$.
Let $A\in \C^{(m-n)\times m}$ be an $(m-n)\times m$-matrix. Each face $\alpha\in K$ defines a submatrix
$A_\alpha\in \C^{(m-n)\times (m-|\alpha|)}$ given by the columns indexed by the elements not in  $\alpha$.
The matrix
$A$ is called {\it $K$-admissible} if for each face $\alpha\in K$ the matrix $A_\alpha$
is an epimorphism.
For example, the matrix $A=(s^r)$,
$1\leq r\leq m,\ 1\leq s\leq m-n$ , is $K$-admissible for any $(n-1)$-dimensional complex
with $m$-vertices.

We define a map
$$
f_A \colon \C^{m}\times Z_K \lra \C^{m-n} \times Z_K
$$
by $f_A(x,z)\letbe (y, z)$ where $y\letbe A(x\quer z)$.
Here, we use that $Z_K\subset \C^m$.
If $T^m$ acts on the target only via the second coordinate, then the space
$E(m-n)\letbe\C^{m-n}\times Z_K$ is the total space of the
$(m-n)$-dimensional trivial $T^m$-bundle $\C^{m-n}\da Z_K$.
The source is the
total space $E(\lambda')$
of the $T^m$-bundle $\lambda'$.

\begin{prop} \label{geoconst}$\ \ \ \ $
Let $K$ be a $(n-1)$-dimensional complex and $A\in \C^{(m-n)\times m}$ be a $K$-admissible matrix.
\nz
(i) The map $f_A \colon E(\lambda') \lra E(m-n)$
is $T^m$-equivariant and a bundle epimorphism.
\nz
(ii) There exists a $T^m$-equivariant $n$-dimensional complex vector bundle $\xi'\da Z_K$ such that
$\lambda'\cong \C^{m-n}\oplus \xi'$ as $T^m$-equivariant vector bundles over $Z_K$.
\nz
(iii) We have $\lambda\cong \C^{m-n}\oplus \xi'_{hT^m}$ as vector bundles over $DJ(K)$.
\nz
(iv) The total Chern class of $\xi'_{hT^m}$ is given by
$c(\xi'_{hT^m})=c(K)\in \Z[K]$.
\nz
(v) The total Pontrjagin  class of $(\xi'_{hT^m})_{\R}$ is given by
$p((\xi'_{hT^m})_{\R})=p(K)\in \Z[K]$.
\end{prop}

\begin{proof}
For $t\in T^m$, we have
$f_A(t(x,z)=f_A(tx, tz)=(A(tx\quer t \quer z,tz)=(A(x\quer z),tz)=tf_A(x,z)$.
This shows that $f_A$ is $T^m$-equivariant. Every point $z=(z_1,...,z_m)\in Z_K$ is contained
in $(S^1)^{[m]-\alpha}\times (D^2)^\alpha$ for some maximal face $\alpha\in K$.
In particular, $z_i\neq 0$ for $i\not\in \alpha$. Since the matrix $A_\alpha$ is invertible,
this shows that the restriction of $f_A$ to the fiber over $z$ is an epimorphism and that
$f_A$ is a
bundle epimorphism.
This proves the first part.

The kernel
$\xi'\letbe \ker f_A$ is again a $T^m$-equivariant vector bundle over $Z_K$.
Since for a compact Lie group $G$
every short exact sequence of $G$-equivariant vector bundles over compact spaces
splits \cite{segal},
we get $\lambda' \cong \xi'\oplus\C^{m-n}$.
This proves the second part. The other three parts are direct consequences
of Part (ii).
\end{proof}

\begin{rem} \label{defxi}
This gives a geometric construction of the vector bundle $\xi\letbe \xi'_{hT^m}$,
whose existence was claimed in Theorem \ref{thm1}. We will give an alternative proof of
Theorem \ref{thm1} in the later sections.

The construction of $\xi$ depends on the chosen matrix $A$. But, as Theorem \ref{thm2} says,
the isomorphism type is independent of $A$.
\end{rem}

\section{Uniqueness properties of $\xi$} \label{global uniqueness}

In this section we describe vector bundles by their
classifying maps. That is an $r$-dimensional complex vector bundle $\eta$ over a space $X$ is a map
$\eta\colon X\lra BU(r)$. Then, Theorem \ref{thm1} and Theorem \ref{thm2}
become statements about existence of maps
and the number of homotopy classes of maps realizing prescribed cohomological data
given by characteristic classes.
As usual, such problems are best approached by considering separately their rational and $p$-adic
versions. An application of Sullivan's arithmetic square will then provide the desired global information.
As a side effect we will also provide a homotopy theoretic construction of
the vector bundle $\xi\colon DJ(K)\lra BU(n)$, already constructed in Section \ref{xi}.

Let $U(r) \lra SO(2r)\lra SO(2r+1)$
denote the standard inclusions of the compact Lie groups $U(r)$, $SO(2r)$ and $SO(2r+1)$.
Composition with $BU(r) \lra BSO(2r)$ induces
the realification of a complex vector bundle $\eta\colon  X \lra BU(r)$
and, as in the introduction, will be denoted by $\eta_\R$. Composition with $BSO(2r) \lra BSO(2r+1)$
means that we add an 1-dimensional trivial bundle.
If we think of $T^r$ as the set of diagonal matrices in $U(r)$,
the compositions
$$
T^r\lra U(r)\lra SO(2r) \lra SO(2r+1)
$$
describe at each stage
the standard maximal torus of the compact Lie group. The integral cohomology
$H^*(BT^r;\Z)\cong \Z[v_1,...v_r]=\Z[r]$ is a polynomial algebra
with $r$ generators in degree 2.
Passing to classifying spaces
and cohomology establishes maps
$$
\begin{array}{rcl}
H^*(BU(r);\Z)&\lra &\Z[r]^{W_{U(r)}} = \Z[c_1,...,c_r]\subset \Z[r]\\
H^*(BSO(2r);\Z)&\lra &\Z[r]^{W_{SO(2r)}} = \Z[p_1,...,p_{r-1},e_r]\subset \Z[r]\\
H^*(BSO(2r+1);\Z)&\lra &\Z[r]^{W_{SO(2r+1)}} = \Z[p_1,...,p_r]\subset \Z[r],
\end{array}
$$
where $W_{U(r)}$, $W_{SO(2r)}$ and $W_{SO(2r+1)}$ denote the Weyl groups,
$c_i$ the universal Chern classes for complex, $p_i$ the universal Pontrjagin classes
for real
and $e_r$ the universal Euler class for oriented $2r$-dimensional real vector bundles.
The class $c_i$ is given by the $i$-th
elementary symmetric polynomial $\sigma_i(v_1,...v_r)$, the class $e_r$ by $\sigma_r(v_1,...v_r)$
and $p_i$ by $\sigma_i(-v_1^2,...,-v_r^2)$. The Euler class satisfies the
equation $e_r^2=(-1)^r p_r$.
In all cases the arrow is an
epimorphism , for $U(r)$ even an isomorphism. If we pass to rational
coefficients all arrows become isomorphisms.

If $K$ is a finite simplicial complex of dimension $n-1$,
the composition
$\Z[c_1,....,c_m] \lra \Z[m] \lra  \z[K]$ maps the Chern class $c_j$ to zero
for $j\geq n+1$.
Hence this composition factors uniquely through $\z[c_1,...,c_n]$ and establishes
a map $f\colon \Z[c_1,...,c_n]\lra \Z[K]$ and a commutative diagram
$$
\diagram
\Z[c_1,...,c_m] \dto \rto & \Z[m] \dto \\
\Z[c_1,...,c_n] \rto^{\ \ f} & \Z[K]
\enddiagram
$$
This diagram
can partly be realized by spaces and maps, namely by
$$
\diagram
BU(m) & \lto BT^m \\
BU(n) \uto & DJ(K) \uto
\enddiagram
$$
We want to construct a map $DJ(K) \lra BU(n)$
making the above diagram
commutative up to homotopy. Such a map is a realizations of $f$ and establishes an
$n$-dimensional complex vector bundle over $DJ(K)$ which has the
desired Chern classes. In fact, the bundle $\xi\da DJ(K)$ constructed in
Section \ref{xi} is such a map. But we will give here another construction.
We also want to show that
such a map is unique up to homotopy.

We can play a similar game with $BSO(2m)$ and $BSO(2n)$.
In this case, the composition $\Z[p_1,....,p_{m-1}, e_m] \lra \Z[m] \lra  \Z[K]$
factors through $\Z[p_1,...,p_{n-1}, e_n] \lra \Z[K]$. We have some freedom
for the choice of the image of the Euler class $e_n$.
Since $p_n$ is mapped to the non trivial class $p_n(K)\letbe(-1)^n\sigma_n(v_1^2,...v_m^2)$ i,
each square root of $(-1)^n p_n(K)$
establishes a different factorization
$\Z[p_1,...,p_{n-1}, e_n] \lra \Z[K]$.
Again, all these factorizations can be realized and are unique up to homotopy.
Before we make these statements precise, we need a description of all square roots.
The Euler classes $e_\omega$ associated to functions $\omega\colon M_K\lra \{\pm 1\}$
and defined in the introduction
provide a complete  list of such square roots.

\begin{lem} \label{eulerclass}
Let $K$ be an $(n-1)$-dimensional simplicial complex. Then,
every square root of
$(-1)^n p_n(K)$ is of the form $e_\omega$.
\end{lem}

\begin{proof}
Every square root $e$ of
$(-1)^n p_n(K)=\sigma_n(v_1^2,...,v_m^2)=\sum_{\mu\in M_K} v_\mu^2$
can be write as a linear combination $e=\sum_{i=1}^r a_iq_i$  of monomials
$q_i$ which are of the form $\prod_{j=1}^m v_j^{l_j}$ with $l_j\geq 0$ and $\sum_j l_j=n$.
We can compare the square $e^2=\sum_i a_i^2q_i^2 +\sum_{i < k} 2a_ia_jq_iq_k$
with $\sum_{\mu\in M_K} v_\mu^2$.
By Remark \ref{limitzk}, both expressions are equal in $\Z[K]$ precisely, when
all non trivial monomials have equal coefficients and
$q_i^2=0$ if and only if $q_i=0$. Hence, if $\mu\in K$ is a maximal face,
the coefficient of $v_\mu$ in $e$ is $\pm 1$ and all other coefficients
vanish.
This shows that $e=e_\omega$ for an appropriate function $\omega\colon M_K\lra \{\pm 1\}$.
\end{proof}

As already explained each class $e_\omega$ establishes a unique
factorization
$$
g_\omega\colon \Z[p_1,...,p_{n-1},e_n] \lra \Z[K].
$$
We say that a map $\rho_\omega\colon DJ(K)\lra BSO(2n)$ realizes
$g_\omega$ if $H^*(\rho_\omega ;\q)=g_\omega\otimes \Q$.

\begin{thm} \label{globalexistence}
Let $K$ be an $(n-1)$-dimensional finite simplicial complex and $\omega\colon M_K\lra \{\pm 1\}$ a function.
Then there exist realizations
$$
\xi\colon DJ(K) \lra BU(n) \ \text{ and }\  \rho_{\omega} \colon DJ(K)\lra BSO(2n)
$$
of
$f\otimes \Q$ and
$g_{\omega}\otimes \Q$
\end{thm}

The uniqueness question can be handled simultaneously for the different
cases. For $n\leq r$ we denote by  $G_r$ one of the groups $U(r)$, $SO(2r)$ or $SO(2r+1)$ and
by $\gamma_r \colon DJ(K)\lra BG_r$ one of the maps
$\xi\oplus \C^{r-n}$, $\rho_{\omega}\oplus\R^{2(n-r)}$ or $\xi_\R \oplus \R^{2(r-n)+1}$.

\begin{thm} \label{globaluniqueness}
Let $K$ be an $(n-1)$-dimensional finite simplicial complex.
Then, a map $\beta_r: DJ(K)\lra BG_r$ is homotopic to $\gamma_r$
if and only if $H^*(\beta_r;\Q)=H^*(\gamma_r;\Q)$.
\end{thm}

The above two results are versions  of our first two main theorems.

{\it Proof of Theorem \ref{thm1} and Theorem \ref{thm2}:}
Theorem \ref{thm1} is contained in Theorem \ref{globalexistence} and Theorem \ref{thm2}(i) and (iii) in
Theorem \ref{globaluniqueness}. The non vanishing of $c_n(\xi)$ and
$p_n(\xi_\R)$ follows from the fact that $K$ contains a face of dimension $n-1$.
The inequalities $r\geq n$ and $s\geq 2n$ follow from the non vanishing of
these Chern and Pontrjagin classes.
Since $DJ(K)$ is simply connected, every real vector bundle $\rho\da DJ(K)$ can be given an
orientation.
If $\dim \rho > 2n$, then $e(\rho)=0=e(\xi_\R \oplus \R^{s-2n})$ and if $\dim \rho=2n$ then
the same formula holds for $\rho\oplus \R$ and $\xi_\R \oplus \R$.
In both cases, Theorem \ref{globaluniqueness} establishes the desired isomorphism
between the vector bundles.
\qed

As already mentioned the proofs of the above theorems are based on an
arithmetic square argument and
rational and $p$-adic versions of the above statements.
In the rest of this section we will reduce the proof of Theorem \ref{globalexistence}
and Theorem \ref{globaluniqueness} to their
$p$-adic versions.

For a  topological space $X$
we denote by $X_0$ the rationalization, by $X\p$ the $p$-adic completion
in the sense of Bousfield and Kan, by $X^\wedge\letbe \prod_p X\p$
the product of all p-adic completions, and by $X_{\af}$
the finite adele type of $X$, which is the formal completion of $X_0$ or the rationalization
of $X^\wedge$.
All spaces under consideration are simply connected.In this case
Sullivan's arithmetic square
$$
\diagram
X \rto \dto & X^\wedge \dto \\
X_0 \rto & X_{\af}
\enddiagram
$$
is a homotopy pull back diagram.

\begin{thm} \label{rationalexistence}
Let $K$ be an $(n-1)$-dimensional finite simplicial complex and
and $\omega\colon M_K\lra \{\pm 1\}$ a function. Then there exist realizations
$$
\xi_0 \colon  DJ(K) \lra BU(n)_0
\ \text{ and }\
(\rho_{\omega})_0\colon DJ(K) \lra BSO(2n)_0
$$
of $f\otimes \Q$ and
$g_{\omega}\otimes \Q$.
\end{thm}

\begin{thm} \label{rationaluniqueness}
Let $K$ be an $(n-1)$-dimensional finite simplicial complex.
A map $(\beta_r)_0\colon DJ(K) \lra (BG_r)_0$ is homotopic to $(\gamma_r)_0$ if and only
if $H^*((\beta_r)_0;\q)=H^*(\gamma_r;\Q)$.
\end{thm}

{\it Proof of both theorems:}
All claims follow from the fact that $(BG_r)_0$ is a product of
rational Eilenberg-Maclane spaces.
\qed

\medskip

For abbreviation, we define $H^*_{\q\p}(-)\letbe H^*(-;\z\p)\otimes_{\z\p}\Q\p$.

\begin{thm} \label{padicexistence}
Let $p$ be a prime,
$K$ an $(n-1)$-dimensional finite simplicial complex,
and $\omega\colon M_K\lra \{\pm 1\}$ a function.
Then there exist realizations
$$
\xi\p\colon  DJ(K) \lra BU(n)\p
\ \text{ and }\
(\rho_{\omega}){}\p\colon DJ(K)\lra BSO(2n)\p
$$
of $f \otimes_{\Z} \Q\p$
and
$g_{\omega}\otimes_{\Z}\Q\p$.
\end{thm}

\begin{thm} \label{padicuniqueness}
Let $p$ be a prime and
$K$ an $(n-1)$-dimensional finite simplicial complex.
A map $(\delta_r)\p\colon DJ(K)\lra (BG_r)\p$ is homotopic
to the completion $(\gamma_r)\p $ of $\gamma_r$ if and only if
$H^*_{\Q\p}((\delta_r)\p)=H^*_{\Q\p}((\gamma_r)\p)$.
\end{thm}

The proof of the last two theorems will be postponed to Section \ref{p-adic case}.

\medskip

The following statement is necessary for the proof of Theorem \ref{globaluniqueness}.

\begin{thm} \label{mono}
The map
$$
[DJ(K),BG_r] \lra [DJ(K),BG_r\com]
$$
is a monomorphism.
\end{thm}

\begin{proof}
Since the arithmetic square for $BG_r$ is a pull back, the homotopy fiber $F$
of $BG_r\lra BG_r\com$ is
equivalent to the homotopy fiber
of the rationalization
$(BG_r)_0 \lra (BG_r)_{\af}$. Since $(BG_r)_0$ is a product of rational
Eilenberg-MacLane spaces of even degrees, $\pi_i(F)=0$ for $i$ even.
The obstruction groups
for lifting homotopies between maps $DJ(K) \lra BG_r^\wedge$ to $BG_r$
are given by $H^*(DJ(K);\pi_*(F))$. All these obstruction groups vanish,
since $H^*(DJ(K);\z)$
is concentrated in even degrees.
\end{proof}

\medskip

{\it Proof of Theorem \ref{globalexistence}.}
Since the rationalization  $BU(n)_0$ is a product of rational Eilenberg-MacLane spaces
of even degrees, the same holds for the finite adele type $BU(n)_{\af}$.
Up to homotopy,  maps into $BU(n)_{\af}$ are determined by cohomological
information. Since $\xi_0$ and $\xi\p$ realize $f\otimes \Q$
respectively $f\otimes \Q\p$, the two compositions
$$
DJ(K) \larrow{\xi_0} BU(n)_0\lra BU(n)_{\af}
$$
and
$$
DJ(K) \larrow{\prod \xi\p} \prod_p BU(n)\p=BU(n)\com \lra BU(n)_{\af}
$$
are homotopic.
Using the homotopy pull back property of the arithmetic square for $BU(n)$ we can construct a map
$\xi\colon  DJ(K) \lra BU(n)$  with the desired cohomological property.
This proves the claim for $BU(n)$..
For $BSO(2n)$ we can argue analogously.
\qed

\medskip

{\it Proof of Theorem \ref{globaluniqueness}.}
Since $H^*(DJ(K);\z)$ is torsion free, the map $H^*(\gamma_r;\q)$ determines
$H^*(\gamma_r;\z)$ as well as $H^*(\rho;\z\p)\otimes \Q\p$.
The homotopical uniqueness of $\gamma_r$ follows from
Theorem \ref{padicuniqueness} and Theorem  \ref{mono}.
\qed

{\it Proof of Theorem \ref{thm3}.}
If the total Pontrjagin class $p(\rho)$ of a oriented vector bundle
$\rho\colon DJ(K) \lra BSO(2n)$ equals $p(K)$, then $e(\rho)^2=(-1)^n p_n(K)$. By
Lemma \ref{eulerclass} there exists a function $\omega\colon M_K\lra \{\pm 1\}$ such that
$e(\rho)=e_\omega$ and by Theorem \ref{thm2} this implies that $\rho$ and $\rho_\omega$
are isomorphic.
\qed

\bigskip

\section{higher limits} \label{higher limits}

The proofs of Theorem \ref{padicexistence} and Theorem \ref{padicuniqueness}
involve the calculation of the some higher derived limits for covariant functors defined
on the opposite category $\catkop$ of $\catk$. For a definition and properties of higher derived
limits see \cite{boka} or \cite{oliver}.
We will drop $\catkop$ from the notation of limits.

Let $\ab$ denote the category of abelian groups. We would like to construct a
filtration for a functor
$\Phi\colon \catkop \lra \ab$. For
$0\leq s\leq n$ we define
functors $\Phi_{\leq s}, \Phi_s\colon  \catkop \lra \ab$ by
$$
\Phi_{\leq s}(\alpha)\letbe
\begin{cases}
\begin{array}{ll}
\Phi(\alpha) & \text{ for } | \alpha|\leq s \\
0 & \text{ for } | \alpha| > s
\end{array}
\end{cases}
\ \ \
\Phi_s
(\alpha)\letbe
\begin{cases}
\begin{array}{ll}
\Phi(\alpha) & \text{ for } | \alpha| = s \\
0 & \text{ for } | \alpha| \neq s
\end{array}
\end{cases}
$$
Since for $|\alpha| < | \beta|$ there is no arrow
$\alpha \ra \beta$
in $\catkop$, both functors are well defined. Moreover,
we have $\Phi_{\leq n}=\Phi$.
There exist short
exact sequences of functors
$$
1\lra \Phi_{\leq s-1} \lra \Phi_{\leq s}\lra \Phi_{s} \lra 1
$$
which induce long exact sequences
$$
...\ra \lim{}^{i-1}\, \Phi_{s} \ra \lim{}^{i}\, \Phi_{\leq s-1}
\ra
\lim{}^{i}\, \Phi_{\leq s}
\ra \lim{}^{i}\, \Phi_{s}
\lra ...
$$
of higher derived limits.

\begin{lem}\label{atomic}$\ \ \ \ \ $
\nz
(i) $\lim{}^{i}\, \Phi_s=0$ for $i\geq n -s +1$.
\nz
(ii)
$\lim{}^{i}\, \Phi_{\leq s} \lra \lim{}^{i}\, \Phi$
is an isomorphism for $s\geq n-i+1 $.
\nz
(iii)
If $i\geq 1$ and if $\Phi(\beta) \cong \Phi(\alpha)$ for $\alpha \subset \beta$ and
$| \beta|\leq n-i+1$ then
$\lim{}^{i}\, \Phi =0$.
\end{lem}

\begin{proof}
For each $s$ the functor
$\Phi_s \cong \prod_{\alpha\in K, | \alpha| = s} \Phi_\alpha$
is a product of atomic functors $\Phi_\alpha$, i.e. $\Phi_\alpha(\beta)= 0$ if
$\beta\neq \alpha$. In \cite{nora2} the higher limits of atomic functors are calculated. We have
$$
\lim{}^{i}\, \Phi_\alpha=
\widetilde{H}^{i-1}(\link_K(\alpha);\Phi(\alpha)),
$$
where
$
\link_K(\alpha)\letbe \{\beta\in K \colon  \alpha\cap\beta=\emptyset, \alpha\cup\beta\in K\}
$
denotes the link of the face $\alpha$ in $K$. In particular,
$\dim(\link_K(\alpha)\leq n-| \alpha| -1$. Hence, these groups vanish
for $i\geq n-|\alpha|+1$. This proves the first part.

Since $\lim{}^{j}\, \Phi_{s+1}=0$ for $j\geq n-s$, the second part
follows from the above long exact sequences
for higher derived limits.

Let $M\letbe \Phi(\emptyset)$ and let $\const_M\colon \catkop \lra \ab$
denote the constant functor with value $M$; i.e. every face is mapped to $M$ and every morphism to
the identity.
Then, $\lim{}^i\, \const_M=0$ for $i\geq 1$
since $\catkop$ has a terminal object and is contractible \cite{boka}.
By part (i) and (ii), we get for $i\geq 1$
\begin{eqnarray*}
0=\lim{}^{i}\, \const_M & \cong &
\lim{}^{i}\, (\const_M)_{\leq n-i+1}\\
& \cong & \lim{}^{i}\, \Phi_{\leq n-i+1}
\cong \lim{}^{i}\, \Phi,
\end{eqnarray*}
which proves the third part.
\end{proof}

\section{The $p$-adic case} \label{p-adic case}

We will work with the model
$hc(K)=\hocolim_{\catk}\, BT^K$.
Again all colimits, limits and higher derived limits are
taken over $\catk$ or
the opposite category $\catkop$. For simplification we will continue to drop
these categories from the notation of limits. We will also skip the
notation ${}\p$ from completion of maps. Again, $G_r$ will denote
either $U(r)$, $SO(2r)$ or $SO(2r+1)$, and
$T^r\subset G(r)$ the standard maximal torus.

Theorem \ref{padicexistence} and Theorem \ref{padicuniqueness} state the existence and
uniqueness of particular maps $hc(K)\lra BG_r{}\p$. We need some tools to calculate
$\pi_0(map(hc(K),BG_r{}\p)$. We have
$$
\begin{array}{rl}
map(hc(K),BG_r{}\p)&\simeq map(\hocolim BT^K,BG_r{}\p)\\
 &\simeq \holim map(BT^K,BG_r{}\p).
\end{array}
$$
This establishes a map
$$
R\colon [hc(K), BG_r{}\p] \lra \lim\,  [BT^K,BG_r{}\p]
$$
where $[-,-]$ denotes the set of homotopy classes of maps.
An element $\hat\phi\in \lim [BT^K,BG_r{}\p]$ is given by a set of homotopy classes of
maps $\phi_\alpha \colon BT^\alpha\lra BG_r{}\p$. We want to know whether
$R^{-1}(\hat\phi)$ is non empty and, if this is the case, how many elements are contained in
$R^{-1}(\phi)$.
The Bousfield-Kan spectral sequence for homotopy inverse limits \cite{boka},
together with work by Wojtkowiak \cite{wojtkowiak}
clarifying the  situation in small degrees, provides an obstruction theory which
may answer both questions.
The obstruction groups are given by the higher derived limits of
the functors
$$
\Pi_i^{G_r}\colon \catkop \lra \ab
$$
which are defined by $\Pi_i^{G_r}(\alpha)\letbe \pi_i(map(BT^\alpha,BG_r{}\p)_{\phi_\alpha})$.
If $\lim^j \Pi_i^{G_r}=0$ for $j=i,\ i+1$ and all $i\geq 1$, then $R^{-1}(\phi)$
consist of exactly one element. That is there
there exists a map $\phi\colon hc(K)\lra BG_r{}\p$, uniquely determined up to homotopy,
such that the restriction
$\phi|_{BT^\alpha}$ is homotopic to $\phi_\alpha$.

For the proof of Theorem \ref{padicexistence} we need to construct particular maps
$\xi\colon hc(K) \lra BU(n)\p$ and $\rho_{\omega} \colon hc(K) \lra BSO(2n)\p$.
We will discuss the construction in detail only for the latter case.
It can be easily adapted to the first case.

For each face $\alpha=\{i_1,...i_t\}\in K$ we want to specify classes in
the polynomial algebra $\Z[\alpha]$.
For $i\leq |\alpha|$ we define $c_i(\alpha)\letbe \sigma_i(v_{i_1},...,v_{i_t})$ and
$p_i(\alpha)\letbe (-1)^i \sigma_i(v_{i_1}^2,...,v_{i_t}^2)$. If $i>|\alpha|$ we set
$c_i(\alpha)=p_i(\alpha)=0$. We also need an Euler class $e(\alpha)$ defined by
$e(\alpha)\letbe c_n(\alpha)$ if $|\alpha|=n$ and $e(\alpha)=0$ if $|\alpha|< n$.
Here, $\sigma_i$ again denotes the $i$-th elementary symmetric polynomial.

Let $\omega\colon M_K\lra \{\pm 1\}$ be a function and $e_\omega$ the
associated square root of $(-1)^np_n(K)$.
The composition
$$
\z[p_1,...,p_{n-1}, e_n] \larrow{g_{\omega}}\z[K]\larrow{h^\alpha} \z[\alpha]
$$
of $h^\alpha$, defined in Remark \ref{limitzk}, and $g_\omega$
maps $p_i$ onto $p_i(\alpha)$ and $e_n$ onto $0$ if $|\alpha|<n$ and onto
$\omega(\alpha) e(\alpha)$ if $|\alpha|=n$.
This algebraic map can be realized by a composition
$$
\rho_{\omega}{}_\alpha\colon  BT^\alpha \lra BT^n \lra BSO(2n),
$$
where the
second map is induced
by the maximal torus inclusion $T^n \subset SO(2n)$
and the first map by a
coordinate wise inclusion $T^\alpha \lra T^n$ combined with complex
conjugations on some of the coordinates.
For $|\alpha|< n$ any number of conjugations is allowed, but for $|\alpha|=n$,
we need an odd number
of complex conjugation if $\omega(\alpha)=-1$
and an even number otherwise.

The Weyl group $W_{SO(2n)}$ of $SO(2n)$ is isomorphic to
$(\z/2)^{n-1}\rtimes \Sigma_n$, where $\Sigma_n$ acts on $T^n$
via permutations of the
coordinates and $(\Z/2)^{n-1}$ via even numbers of complex conjugation on coordinates.
In fact, for each pair of coordinates, there exists an element in $(\Z/2)^{n-1}$ which induces
complex conjugation exactly
on these two coordinates.
For $i=1,2$ let $\phi_i \colon T^\alpha \lra T^n$ be a coordinatewise inclusion
combined
with $l_i$ complex conjugations. If $|\alpha|=n$ then there exists $w\in W_{SO(2n)}$
such that $\phi_1=w\phi_2$
precisely when
$l_1$ and $l_2$ are congruent modulo 2. And if $|\alpha|<n$, then there
exists a missing coordinate which we can use to pass from an even number
to an odd number of complex conjugations. In this case, there always exists $w\in W_{SO(2n)}$ such that
$\phi_1=w\phi_2$.
This shows that the underlying homomorphisms for different choices
of $\rho_{\omega}{}_\alpha$ are conjugate in $SO(2n)$ and that different choices produce homotopic
maps between the classifying spaces.

This argument also shows that, if
$\alpha \subset \beta$, the triangle
$$
\diagram
BT^\alpha \rrto \drto^{\rho_{\omega}{}_\alpha} & & BT^\beta \dlto_{\rho_{\omega}{}_\beta}\\
&BSO(2n)\p
\enddiagram
$$
commutes up to homotopy.
In particular the set of maps $\rho_{\omega}{}_\alpha$ defines an element
$\hat \rho_{\omega}\in [BT^K,BSO(2n)\p]$.
If $\lim{}^{j}\, \Pi_i^{SO(2n)}=0$ for $j=i,\ i+1$ and all $i\geq 1$ then there exists a map
$\rho_{\omega}\colon hc(K)\lra BSO(2n)\p$, uniquely determined up to homotopy such that
$\rho_{\omega}|_{BT^\alpha}\simeq \rho_\omega{}_\alpha$.

\begin{exa}
We illustrate the above construction with an example.
Let $K$ be the boundary of the 2-dimensional simplex $\Delta[3]$,
that is $\{1,2,3\}$ is the only missing face.  Then, the dimension of $K$ is 1
and the Stanley-Reisner algebra is given by
$\Z[K]=\Z[3]/(v_1v_2v_3)$. We fix an Euler class $e_\omega\letbe -v_1v_2-v_2v_3+v_1v_3$
and want to construct a map $\rho\colon DJ(K)\lra BSO(4)$ such that
$\rho^*(p_1)=p_1(K)=-(v_1^2+v_2^2+v_3^2)$
and $\rho^*(e_1)=e_\omega$.

The first factor of the  Weyl group $W_{SO(4)}\cong \z/2\times \z/2$.
 acts on the maximal torus $T^2$ by complex conjugation on both coordinates and
the second factor permutes
both coordinates. For each $\alpha\in K$ we define a map
$\iota_\alpha\colon T^\alpha\lra T^2$ as  follows. We set
$\iota_{\{i\}}(z_i)\letbe (z_i,1)$ for $i=1,2,3$,
$\iota_{\{1,2\}}(z_1,z_2)\letbe (\quer z_1,z_2)$,
$\iota_{\{1,3\}}(z_1,z_3)\letbe (\quer z_1,\quer z_3)$
and $\iota_{\{2,3\}}(z_2,z_3)\letbe (\quer z_2,z_3)$.
The maps $\rho_\alpha$ are given by the p-adic completion of the map $BT^\alpha\lra BSO(4)$
induced by
the composition of $\iota_\alpha$ and
the maximal torus inclusion
$T^2\lra SO(4)$.
For maximal faces $\alpha\in K$ the pull back of the universal Euler and first Pontrjagin class along $\rho_{\alpha}$
are given by
$(\rho_{\{1,2\}})^*(e_2) = -v_1v_2$,
$(\rho_{\{1,3\}})^*(e_2) = -v_1v_3$ and
$(\rho_{\{2,3\}})^*(e_2) = v_2v_3$.

The restriction $\iota_{\{1,3\}}|_{T^{\{3\}}}$ and $\iota_{\{3\}}$ differ
by a permutation and complex conjugation of both coordinates, are conjugate in $SO(4)$
and establish homotopic maps $\rho_{\{1,3\}}|_{BT^{\{3\}}}$ and $\rho_{\{3\}}$.
Similarly, for all other choices of $j\in \alpha$, $|\alpha|=2$,
we get $\rho_\alpha|_{BT^{\{j\}}}\simeq \rho_{\{j\}}$.
The collection of maps $\rho_\alpha$ defines an element in $\lim[BT^K,BSO(4)\p]$.
If all obstruction groups vanish,
the maps $\rho_\alpha$ fit
together to define a map
$\rho_\omega \colon DJ(K)\lra BSO(4)\p$ and we have $\rho_\omega^*(p_1)=p_1(K)$ and
$\rho_\omega^*(e)=e_\omega$.
\end{exa}

In the case of $U(n)$ we do not need to care about complex conjugation.
The composition $\Z[c_1,...,c_n] \larrow{f} \z[K]\larrow{h^\alpha} \Z[\alpha]$ maps
$c_i$ onto $c_i(\alpha)$.
The map $\xi_\alpha\colon BT^\alpha \lra BU(n)\p$ is induced by
the composition $T^\alpha\larrow{f}  T^n\larrow{h_\alpha} U(n)$
of some coordinate wise inclusion followed by the
maximal torus inclusion. In this case the obstruction groups are given by higher derived limits of
the functors $\Pi^{U(n)}_i$.
The vanishing of the higher derived limits will establish a map
$\xi\colon hc(K) \lra BU(n)\p$, uniquely determined up to homotopy,
such that $\xi|_{BT^\alpha}=\xi_\alpha$.

For $n\leq r$ and a face $\alpha\in K$, we  denote by $\gamma_r{}_\alpha: BT^\alpha \lra BG_r$ one of the
maps $\xi_\alpha\oplus \C^{r-n}$, $\rho_{\omega}{}_\alpha\oplus\R^{2(r-n)}$ or $(\xi_\alpha)_\R\oplus \R^{2(r-n)+1}$.
By the above construction these maps fit together to define an element in
$\lim\, [BT^K,BG_r]$ and define functors $\Pi_i^{G_r} :\catkop \lra \ab$.

\begin{prop} \label{vhl}
For all $j\geq i\geq 1$,
we have
$$
\lim{}^j\, \Pi_i^{G_r}=0
$$
\end{prop}

The proof needs some preparation.
The involved mapping spaces can be described in terms of centralizers of subtori.
A map $\gamma_r{}_\alpha$
is induced by a homomorphism $\iota_\alpha \colon T^\alpha \lra G_r$ which is given by the composition of
a coordinate wise  inclusion $T^\alpha \lra T_{G(r)}$,
possibly some conjugations on coordinates, and the inclusion of the maximal torus
into $G(r)$. Conjugation on coordinates of the maximal torus
has no effect on the centralizer of $\iota_\alpha$. The centralizer
$C_{G_r}(T^\alpha)\letbe C_{G_r}(\alpha)$ of the image of $\iota_\alpha$
can be identified with  $T^\alpha\times G_{r-|\alpha|}\subset G_r$,
where the inclusions of the factors are induced by the subset
relations $\alpha\subset [r]$ and $[r]\setminus\alpha\subset [r]$.

By \cite{no1}, there exists a mod-p equivalence
$$
BT^\alpha\times  BG_{r-|\alpha|}= BC_{G_r}(j_\alpha)
\lra map(BT^\alpha,BG_r{}\p)_{\gamma_\alpha}.
$$
Moreover, up to homotopy the above mod-p equivalence is
natural with respect to  morphisms $\beta \supset \alpha$ in $\catkop$.
Such an inclusion  induces the composition
$$
BT^\beta \times BG_{r-|\beta|}\lra
BT^\alpha \times BT^{\alpha\setminus \beta}\times BG_{r-|\beta|}
\lra BT^\alpha\times BG_{r-|\alpha|}.
$$
between the classifying spaces of the centralizers. After $p$-adic completion
this map is homotopic to the induced map between
the associated mapping spaces. Passing to homotopy groups, this
describes the map
$$
\Pi_i^{G_r}(\alpha)\lra \Pi_i^{G_r}(\beta).
$$

It will be convenient to define functors
$$
\Psi_2,\hat \Pi_2^{G_r} \colon \catkop \lra \ab
$$
by $\Psi_2(\alpha)\letbe \pi_2((BT^\alpha)\p)$ and
$\hat \Pi_2^G(\alpha)\letbe \pi_2(BG_{(r-|\alpha|)}{}\p)$.

In the following we partly need to distinguish between $U(r)$ on the one hand and
$SO(2r)$ and $SO(2r+1)$ on the other hand. We will denote by $SO_r$ one of
special orthogonal groups.

\begin{lem} \label{necessary}$\ \ \ \ \ $
\nz
(i) For $i=2$, we have an exact sequence
$$
0\lra \hat \Pi_2^{G_r} \lra \Pi_2^{G_r} \lra \Psi_2 \lra 0
$$
of functors. Moreover, we have $\Psi_2\cong H^2(BT^K;\z\p)$.
\nz
(ii)
$\Pi_{2j+1}^{U(r)}(\alpha)=0$ for all $j\leq r-|\alpha|-1$.
\nz
(iii) Let $\alpha \subset \beta$ and $|\beta| \leq n-j$.
If $j\geq 2$, then
$\Pi_{2j}^{U(r)}(\beta)\cong \Pi_{2j}^{U(r)}(\alpha)$.
If $j=1$, the same formula holds
for $\hat\Pi_2^{G_r}$.
\nz
(iv) $\Pi^{G_r}_1(\alpha)=0$ for all $\alpha$.
\nz
(v)If $p$ is odd, then $\hat \Pi^{SO_r}_2(\alpha)=0$
for $|\alpha|\neq r-1$.
\nz
(vi)
If $t\geq 3$ and $\alpha\subset \beta$, then
$\Pi^{SO_r}_t(\beta)\lra \Pi^{SO_r}_t(\alpha)$ is an isomorphism for
$|\beta| \leq r-t/2-1/2$.
If $p=2$, the same formula holds for
$\hat \Pi^{SO_r}_2$.
\end{lem}

\begin{proof}
The natural transformations $\hat \Pi_2^{G_r} \lra \Pi_2^{G_r}$ and
$\Pi_2^{G_r}\lra \Psi_2$ establish the exact sequence of Part (i).
The second claim of $(i)$ is obvious.

Part (ii) and Part (iii) follow from the fact that
$\pi_{2s+1}(BU(t))=0$ for $0\leq s\leq t-1$ and
$\pi_{2s}(BU(t))\cong \pi_{2s}(BU(t+1))$ for $1\leq s\leq t$.

For every connected compact Lie group $H$, the classifying space
$BH$ is simply connected. This proves  Part (iv).

The fifth part follows from the fact that $\pi_2(BSO(s))\cong \z/2$ for
$s\geq 3$.

Finally, $\pi_t(BSO(k))\lra \pi_t(BSO(k+1))$ is an isomorphism for
$t\leq k-1$. Since $\Pi_t^{SO_r}(\beta)=\pi_t(BSO_{n-|\beta|}{}\p)$
and $\hat \Pi_2^{SO_r}(\beta)=\pi_2(BSO_{n-|\beta|})\p)$,
this implies the last claim.
\end{proof}

{\it Proof of Proposition \ref{vhl}:}
Since $\Pi_1^{G_r}\equiv 0$ (Lemma \ref{necessary}(iv)),
we have $\lim{}^{j}\, \Pi_1^{G_r}=0$ for $j\geq 1$.

Now we consider the functor $\Pi_i^{U(r)}$ for $i\geq 3$.
If $i=2k+1$ and $|\alpha| \leq n-i+1=n-2k\leq n-k-1$, then
$\Pi_{2k+1}^{U(r)}(\alpha)=0$ (Lemma \ref{necessary}(ii)).
Lemma \ref{necessary}(iii) implies that
$\lim{}^{j}\, \Pi_i^{U(r)}=0$  for $j\geq i$.

If $i=2k\geq 3$ and $| \alpha| \leq n-2k\leq n-k-1$, then
$\Pi_i^{U(r)}(\alpha)=\z\p$ and $\Pi_i^{U(r)}(\alpha)\cong \Pi_i^{U(r)}(\beta)$ for
$\beta\subset \alpha$ by Lemma \ref{necessary}(iii). Hence, by Lemma \ref{atomic}(iii),
we have again $\lim{}^{j}\, \Pi_i^{U(r)}=0$ for $j\geq i$.

The same argument shows that $\lim{}^{j}\, \hat \Pi_2^{U(r)}=0$ for
$j\geq 2$.

By Lemma \ref{necessary}(i) the functor $\Pi_2^{U(r)}$ fits into the exact
sequence
$$
0\lra \hat \Pi_2^{U(r)} \lra \Pi_2^{U(r)} \lra \Psi_2\cong H^2(BT^K;\z\p)\lra 0.
$$
For $j\geq 2$, we have
$\lim{}^{j}\, H^2(BT^K;\z\p)=0$ \cite{nora1} as well as
$\lim{}^{j}\, \hat\Pi_2^{U(r)}=0$. Hence, the same holds for $\Pi_2^{U(r)}$.
This finishes the proof in the complex case.

Now we consider the cases of the special orthogonal groups.
If $j\geq i\geq 3$ we have
$n-j+1\leq n-i+1\leq n-i/2-1/2$. This shows that
$\Pi^{SO_r}_i(\beta)\cong \Pi^{SO_r}_i(\alpha)$ for $\alpha\subset \beta$ and
$|\beta| \leq n-j+1$ (Lemma \ref{necessary}(vi)) and that
$\lim{}^j\, \Pi^{SO_r}_i=0$
(Lemma \ref{atomic}(iii).

For $i=2$ we again first consider the functor $\hat\Pi^{SO_r}_2$.
If $p$ is odd, then $\hat \Pi^{SO_r}_2(\alpha)\neq 0$ if and only
if $|\alpha|=r-1$ and $SO_r=SO(2r)$. Hence, if $r\geq n+2$ or $SO_r=SO(2r+1)$,
the functor $\hat\Pi_2^{SO_r}$ is trivial and $\lim^j \hat\Pi_2^{SO_r}=0$.
For $r=n,\ n+1$ and $SO_r=SO(2r)$, the functor $\hat \Pi^{SO_r}_2$ is a product of atomic functors
and only nontrivial either on faces of order $n$ or $n-1$.
But in both case, we have $\lim^j \hat \Pi^{SO_r}_2=0$ for $j\geq 2$ (Lemma \ref{atomic}(i)).

If $p=2$, then
$$
\hat \Pi_2^{SO_r}(\alpha)\cong \pi_2(BSO_{r-|\alpha|{}^\wedge_2})\cong
\begin{cases}
0 &\text{ for } r=|\alpha|\\
\Z^\wedge_2 &\text{ for } r=|\alpha|+1, SO_r=SO(2r) \\
\Z/2 &\text{ otherwise }
\end{cases}.
$$
If $r\geq n+1$ or if $r\geq n$ and $SO_r=SO(2r+1)$, we have
$\hat \Pi^{SO_r}_2(\beta) \cong \hat \Pi^{SO_r}_2(\alpha)$
for $\alpha\subset \beta$ and $|\beta|\leq n-2+1$. Using Lemma \ref{atomic}(iii) again, this shows that
in these cases
$\lim^j \hat \Pi^{SO_r}_2=0$ for $j\geq 2$.
If $r=n$, then we have an exact sequence of functors
$$
0\lra \Phi \lra \hat \Pi^{SO(2r)}_2 \lra \hat \Pi^{SO(2r+1)}_2 \lra 0
$$
where the natural transformation $\hat \Pi^{SO(2r)}_2 \lra \hat \Pi^{SO(2r+1)}_2$
is induced by the inclusion $SO(2r)\ra SO(2r+~1)$ and is an epimorphism.
For the kernel $\Phi$ we have $\Phi=\Phi_{n-1}$, which implies that $\lim^s \Phi=0$
for $s\neq 1$ (Lemma \ref{atomic}(i)).
This shows that $\lim^j \hat \Pi^{SO(2r)}_2=0$ for $j\geq 2$ as well.

Now, the same argument as for $\Pi^{U(r)}_2$ shows that $\lim{}^j\, \Pi^{SO_r}_2=0$.
This finishes the proof for $SO_r$.
\qed

\medskip

The above arguments and Proposition \ref{vhl} allow to draw the following corollary.

\begin{cor} \label{restrictiontotori}
Let $\phi, \psi\colon  DJ(K) \lra  BG_r{}\p$ be two map such that for each face
$\alpha$ of $K$ the restrictions $\phi|_{BT^\alpha}$ and $\psi|_{BT^\alpha}$ are homotopic to
$\gamma_r{}_\alpha$. Then, $\phi$ and $\psi$ are homotopic.
\end{cor}

The following two facts may be found in \cite{no1}.

\begin{thm} \label{facts}
Let $G$ be a connected compact Lie group with maximal torus $T_G\subset G$
and $S$ a torus.
\nz
(i) The map $\hom(S,T_G)/W_G \lra [BS,BG]$ is a bijection.
\nz
(ii) The map $[BS,BG]\lra \Hom(H^*(BG;\Q),H^*(BS;\q))$ is an injection.
\end{thm}

\begin{rem} \label{torimono}
For the proof of Theorem \ref{padicexistence} and
Theorem \ref{padicuniqueness} we need a $p$-adic version of the second part,
That is that
$$
[BT^s,BG\p]\lra Hom(H^*_{\Q\p}(BG),H^*_{\Q\p}(BT^s))
$$
is a monomorphisms.
A similar argument as in the global case will work.

Since $BT_G{}\p$ is an Eilenberg-Maclane space with torsion free
$p$-adic cohomology, the statement holds if we replace $BG$ by $BT_G$.
Moreover, up to homotopy, every map
$BT^s\lra BG\p$ factors through the map $BT_G{}\p\lra BG\p$ \cite[Proposition 8.11]{dwannals}.
Factorizations $BT^s\lra BT_G{}\p$ of a map $\phi\colon BT^s \lra BG\p$
define extensions $H^*_{\q\p}(BT_G)\lra H^*_{\q\p}(BT^s)$  of
$H^*_{\Q\p}(\phi)\colon H^*_{\q\p}(BG)\lra H^*_{\q\p}(BT^s)$. But two extensions differ
only by an element of the Weyl group
$W_G$ of $G$. This is proved in \cite[Theorem 1.7]{am} for rational cohomology, but
since $H^*_{\q\p}(BG)\cong H^*_{\q\p}(BT_G)^{W_G}$ and since $\q\p$ is an infinite field
the same argument works in our case.

Now we can argue as follows. Let $\phi, \psi\colon BT^s \lra BG\p$ be two maps
which induce the same map in $H^*_{\Q\p}(-)$-cohomology. We choose
lifts $\hat \phi,\hat \psi\colon BT^s\lra BT_G{}\p$. By the above arguments there
exists $w\in W_G$ such that $\hat\phi\simeq w\hat \psi$. Since conjugation by elements
induces self maps of $BG$ homotopic to the identity, this shows that
$\phi$ and $\psi$ are homotopic.
\end{rem}

{\it Proof of Theorem \ref{padicexistence}:}
By Proposition \ref{vhl} and the preceding arguments,
there exist maps $\xi\colon DJ(K)\lra BU(n)\p$
and $\rho_{\omega}\colon DJ(K)\lra BSO(2n)\p$
such that, for all faces $\alpha\in K$, the restrictions
$\xi|_{BT^\alpha}\simeq \xi_\alpha$ and $\rho_{\omega}|_{BT^\alpha}\simeq \rho_{\omega}{}f_\alpha$.
Since $\z\p[K]\lra \prod_{\alpha\in K}\z\p[\alpha]$ is a monomorphism, we have
$\xi^*(c_i)=c_i(K)$, $\rho_{\omega}^*(p_i)=p_i(K)$ and $\rho_{\omega}^*(e)=e_\omega$.
\qed.

\medskip

{\it Proof of Theorem \ref{padicuniqueness}:}
We continue to denote by $\gamma_r\colon DJ(K)\lra BG_r{}\p$ one of the maps
$\xi\oplus \C^{r-n}$, $\rho_{\omega}\oplus \R^{2(r-n)}$ or $\xi_{\R}\oplus \R^{2(r-n)+1}$.
Let $\delta_r\colon DJ(K)\lra (BG_r)\p$ be a map such that
$H^*_{\q\p}(\delta_r)=H^*_{\Q\p}(\gamma_r)$. Then, the restrictions
$(\delta_r)_\alpha\letbe \delta_r|_{BT^\alpha}$ and $\gamma_r|_{BT^\alpha}\simeq \gamma_r{}_\alpha$
induce the same map in $H^*_{\q\p}(-)$-cohomology and are therefore homotopic
(Proposition \ref{torimono}). Corollary \ref{restrictiontotori}
shows that $\delta_r$ and $\gamma_r$ are homotopic.
\qed

\section{Complex structures for vector bundles} \label{cxstructure}

The stable uniqueness result for vector bundles over $DJ(K)$
realizing the total Pontrjagin class
$p(K)=\prod_i(1-v_i^2)$
shows that all the bundles $\rho_{\omega}\colon DJ(K)\lra BSO(2n)$
of Theorem \ref{globalexistence} admit stably a complex structure. In fact,
$\rho_{\omega}\oplus \R^2\cong \xi_\R\oplus \R^2$ already
admits a complex structure.
In this section we will discuss the unstable version, namely which of the bundles $\rho_\omega$ actually are
isomorphic to the realification of a complex vector bundle.
We also will classify the complex structures of $\rho_\omega$ up to isomorphism, in the unstable case as well
as in the stable case.
For this purpose we need to classify all complex bundles $\eta$ over $DJ(K)$ with $p(\eta_{\R})=p(K)$.

For every function $f\colon [m] \lra \{\pm 1\}$ we define an element
$$
c_f(K)\letbe \prod_{i=1}^m (1+f(i)v_i)\in \Z[K]=H^*(DJ(K);\z).
$$
We will construct a complex vector bundle
$\xi_f\da DJ(K)$ with $c(\xi_f)=c_f(K)$. If $f(i)=1$ for all $i\in [m]$, then $\xi_f=\xi$.

For each $\alpha\in K$ we define a homomorphism $\iota_{f,\alpha}\colon T^\alpha \lra T^\alpha$
which is given by complex conjugation on  the $i$-th coordinate if $f(i)=-1$ and the identity otherwise.
The collection of the maps
$\theta_{f,\alpha} :BT^\alpha \lra BT^\alpha$ induced by $\iota_{f,\alpha}$
 provides a natural transformation
$\Theta_f\colon BT^K \lra BT^K$
and therefore a map
$$
\theta_f\colon DJ(K)\simeq c(K)\lra c(K)\simeq DJ(K).
$$
By construction, the square $\theta_f^2$ is homotopic to the identity on $DJ(K)$.
In cohomology, the induced map $\theta_f^*\colon \Z[K] \lra \Z[K]$
maps $v_i$ to $f(i)v_i$.
The  total Chern class of the composition $\xi_f\letbe \xi\theta_f\colon DJ(K)\lra BU(n)$
is given by $c(\xi_f)=\prod_{i=1}^m (1+f(i)v_i)=c_f(K)$, the Euler class of $(\xi_f)_\R$
by $e((\xi_f)_\R)=c_n(\xi_f)=\sum_{\mu\in M_K} \omega_f(\mu)v_\mu=e_f(K)$ and the total Pontrjagin class
by $p((\xi_f)_\R)=p(\xi_f)=p(K)$.
As already explained in the introduction, the function $\omega_f\colon M_K\lra \{\pm 1\}$ is given by
$\omega_f(\mu)=\prod_{i\in \mu} f(i)$.

\begin{prop} \label{classcxstructure}
Let $f,g \colon [m]\lra \{\pm 1\}$ be two functions.
\nz
(i) The complex vector bundles $\xi_f$ and $\xi_g$ are isomorphic if and only if $f=g$
\nz
(ii) The realifications  $(\xi_f)_{\R}$ and $(\xi_g)_{\R}$ are isomorphic as non-oriented bundles
if and and only if
$\omega_f=\pm \omega_g$.
\end{prop}

\begin{proof}
If $\xi_f\cong \xi_g$, then both have the same Chern classes given by
$\prod_i (1+f(i)v_i)$ and $\prod_i (1+g(i)v_i)$.
A comparison of the first Chern class of both bundles already shows that $f(i)=g(i)$ for all $i$.

Since the Euler classes of $\xi_f$ and $\xi_g$ are given by $e_f(K)$ and $e_g(K)$. The bundles
$\xi_f$ and $\xi_g$ are isomorphic as non-oriented real vector bundles if and only if the Euler classes of both bundle
differ at most by a sign (Theorem \ref{thm2} and Remark \ref{nonoriented}).
And this is true if and only if $\omega_f=\pm \omega_g$.
\end{proof}

The next theorem classifies all complex bundles whose realification realizes the Pontrjagin class
$p(K)=\prod_i(1-v_i^2)$.

\begin{thm} \label{classification}
Let $K$ be an $(n-1)$-dimensional simplicial complex and
$\eta\colon DJ(K)\lra BU(n)$
a complex vector bundle.
\nz
(i) For any map $f\colon [m]\lra \{\pm 1\}$ the map $\eta$ is homotopic to $\xi_f$ if and only if
$H^*(\eta;\Q)=H^*(\xi_f;\Q)$.
\nz
(ii) If $p(\eta_{\R})=p(K)$, then there exists a
function $f\colon [m]\lra \{\pm 1\}$ such that
$\eta\simeq \xi_f$.
\end{thm}

\begin{proof}
The first part follows from the fact that $\theta_f^2=id_{DJ(K)}$. If $\eta$ and $\xi_f$
induce the same map in rational cohomology, then the same holds
for $\eta\theta_f$ and $\xi_f \theta_f\simeq \xi\theta_f^2\simeq \xi$.
Hence, by Theorem \ref{thm2}, both maps
are homotopic as well as $\eta$ and $\xi_f$.

Now let $\eta\colon DJ(K)\lra BU(n)$ be a complex vector bundle with
$p(\eta_{\R})=\prod_{i=1}^m (1-v_i^2)$. The Euler class $e(\eta_{\R})$
is a square root of $(-1)^np_n(K)$ and determines a function
$\omega: M_K\lra \{\pm 1\}$ such that
$e(\eta_{\R})=e_\omega(K)$.
By Theorem \ref{globaluniqueness}, the maps $\eta_{\R}$ and $\rho_{\omega}$ are homotopic.

For each $\alpha\in K$, the composition
$$
\eta_\alpha\colon BT^\alpha \lra DJ(K)\larrow{\eta} BU(n)
$$
is homotopic to a map induced by
a homomorphism $\iota_\alpha\colon T^\alpha \lra T^n$ composed with
the maximal torus inclusion $T^n \lra U(n)$ (Theorem \ref{facts}).
And the composition
$$
BT^\alpha\lra DJ(K)\larrow{\eta} BU(n)\lra BSO(2n)
$$
is homotopic to a map
induced by $\iota_\alpha$ composed with $T^n\lra U(n)\lra SO(2n)$.
By construction,  the restriction $\rho_{\omega}|_{BT^\alpha}: BT^\alpha\lra BSO(2n)$ is
homotopic to map induced by a coordinate wise inclusion
$T^\alpha\lra T^n$ followed by complex conjugation on some coordinates and composed with
$T^n\lra SO(2n)$. We can apply again Theorem \ref{facts}. This shows that
$\iota_\alpha\colon T^\alpha \lra T^n$ is of the same form, namely coordinate wise
inclusion followed by complex conjugation on some coordinates.
If $\beta\subset \alpha$, the composition
$\iota_\alpha |_{T^\beta}\colon T^\beta\lra T^\alpha\larrow{\iota_\alpha} T^n$ and the homomorphism
$\iota_\beta\colon T^\beta\lra T^n$ differ only by an element of the Weyl group of $U(n)$, that is by a
permutation of the coordinates. Such an operation has no effect on complex conjugations.
This allows to define
a map $f\colon [m] \lra \{\pm 1\}$ by $f(i)=-1$ if the homomorphism
$\iota_{\{i\}}\colon T^{\{i\}} \lra T^n$ involves complex conjugation and $f(i)=1$ otherwise.

By construction of the function $f$, we have $\xi_f|_{BT^\alpha}\simeq \eta|_{BT^\alpha}$ for
all $\alpha\in K$ (Theorem \ref{facts}).
In particular, the maps $\eta$ and $\xi_f$ induce the same map in rational cohomology and are therefore homotopic
(Part (i)).
\end{proof}

\begin{rem} \label{stableclass}
Proposition \ref{classcxstructure} and Theorem \ref{classification} also have stable versions.
The same arguments as above work. Suppose that $t\geq 1$.

\smallskip

\noindent
(i)
For two function $f,g\colon [m]\lra \{\pm 1\}$, the bundles $\xi_f\oplus \C^t$ and
$\xi_g\oplus \C^t$ are isomorphic if and only if $f=g$,
but the realifications $(\xi_f)_\R\oplus \R^{2t}$ and
$(\xi_g)_\R\oplus \R^{2t}$ are always isomorphic.

\smallskip

\noindent
(ii)
Let $\eta\colon DJ(K)\lra BU(n+t)$ be an $(n+t)$-dimensional complex vector bundle.
Then, $\eta$ and $\xi_f\oplus \C^t$ are homotopic if and only if both induce the
same map in rational cohomology. And if $p(\eta_\R)=p(K)$, then there exists a function
$f\colon [m]\lra \{\pm 1\}$ such that $\eta\cong \xi_f\oplus \C^t$.
\end{rem}

Now we are in the position to  prove Theorem \ref{thm4} and Corollary \ref{cor7}, which we recall both
for the convenience of the reader.

\medskip

\noindent {\bf Theorem \ref{thm4}.}
\emph{
Let $K$ be an $(n-1)$-dimensional simplicial complex.
Let $\rho\colon DJ(K) \lra BSO(2n)$ be an oriented real vector bundle such that $p(\rho)=p(K)$.
Then, the vector bundle $\rho$ admits a complex structure
if and only if there exists a function $f\colon [m] \lra \{\pm 1\}$ such that
$e(\rho)=\pm e_f(K)$.
}

\medskip

\noindent
{\bf Corollary \ref{cor7}.}
\emph{
Let $K$ be an $n-1$-dimensional abstract finite simplicial complex and let $s\geq n$. Let $\rho$ be an
$2s$-dimensional
oriented real vector bundle over $DJ(K)$ such that $p(\rho)=p(K)$.
\nz
(i) If $s=n$, the number of isomorphism classes of
complex structures for $\rho$ equals the number of functions $f\colon [m]\lra \{\pm1\}$ such that
$e_f(K)=\pm e(\rho)$. In particular, there exists only a
finite number of isomorphism classes of complex structures for $\rho$.
\nz
(ii) If $s>n$ then there exist exactly $2^m$ non isomorphic complex structures for $\rho$.
}

\medskip

\emph{Proof of Theorem \ref{thm4}}
An even dimensional orientable real vector bundle $\rho$ over $DJ(K)$ admits two orientations.
And switching the orientation means to multiply the Euler class by $-1$. We denote the bundle with the opposite
orientation by $\quer \rho$. In particular, $\rho$ and $\quer \rho$ are isomorphic as non-oriented real vector
bundles.

If the Euler class $e(\rho)$ of $\rho$ has the required form then either $\rho$ or  $\quer \rho$ and $(\xi_f)_\R$
have the same Euler class and the same
Pontrjagin classes. By Theorem \ref{globaluniqueness}, this shows that either $\rho\simeq (\xi_f)_\R$
or that $\quer\rho\simeq (\xi_f)_\R$. In particular,
$\rho$ admits a complex structure.

If $\rho$ admits a complex structure, then there exists a complex vector bundle $\eta\colon DJ(K)\lra BU(n)$
such that $\rho$ and $\eta_\R$ are isomorphic as non-oriented bundles.
In particular, $p(\eta_\R)=p(\rho)=\prod_i(1-v_i^2)$.
By Theorem \ref{classification},
there exists a function $f\colon [m]\lra \{\pm 1\}$ such that $\eta\simeq \xi_f$. And this implies
that $e(\rho)=\pm e((\xi_f)_\R)=\pm e_f(K)$.
\qed

\medskip

\emph{Proof of Corollary \ref{cor7}.}
First we assume that $s=n$.
If $\eta$ is a complex structure for $\rho$, then, by Theorem \ref{classification}, there exists a function
$f\colon [m]\lra \{\pm 1\}$ such that $\eta\cong\xi_f$. Moreover, we have $e_f(K)=e(\eta_\R)=\pm e(\rho)$.
And two functions $f,g\colon [m]\lra \{\pm 1\}$ induce isomorphic vector bundles $\xi_f$ and $\xi_g$
if and only if $f=g$ (Proposition \ref{classcxstructure}). This proves the first part.

If $s>n$, then, for every function $f\colon [m]\lra \{\pm 1\}$ the two bundles
$(\xi_f)_\R\oplus\R^{2(s-n)}$ and $\rho$ have the same Pontrjagin classes and are
isomorphic (Theorem \ref{thm2}). Moreover, Remark \ref{stableclass} tells us that every
complex structure of $\rho$ is of the form $\xi_f\oplus\C^{s-n}$ and that
$\xi_f\oplus \C^{s-n} \cong \xi_g\oplus \C^{s-n}$ if and only iif $f=g$. This completes
the proof of the second part.
\qed

\medskip

\begin{exas} \label{cxstructurexample}
We illustrate our results with some examples.
If $K\letbe\partial\Delta[3]$ is a triangle, we have three vertices, three maximal faces given by
$\letbe \{1,2\}$\ $\{2,3\}$ and $\{1,3\}$, and eight choices for the Euler class $e_\omega(K)$, namely
$\pm v_1v_2\pm v_2v_3\pm v_1v_3$.
If the number of negative signs is even, then there exists a function $f\colon [3]\lra \{\pm 1\}$ such that
$e_\omega(K)=\omega_f(\{1,2\})v_1v_2 + \omega_f(\{2,3\})v_2v_3 + \omega_f(\{1,3\})v_1v_3=c_3(\xi_f)$. In this case,
$\rho_{\omega}\cong (\xi_f)_\R$ as oriented bundles. Since $\dim K$ is odd, faces of maximal order contain
an even number of vertices and
$e_f(K)=e_{-f}(K)$. In particular, there exist two different complex structures for $\rho_\omega$ in this case.
If the number of signs in $e_\omega(K)$ is odd, such a function does not exists and $\rho_{\omega}$ is not
isomorphic to the realification of a complex bundle
as oriented bundle. Since the number of maximal faces is odd, one of the two orientations on $\rho$ satisfies the condition on the signs and every
bundle $\rho_{\omega}$ always admits exactly two non isomorphic complex structure.

If $K$ is a square, we have four vertices and
four maximal faces given by $\{1,2\}$, $\{2,3\}$, $\{3,4\}$ and $\{1,4\}$.
Every function $f\colon [4]\lra \{\pm 1\}$ produces an Euler class $\pm v_1v_2\pm v_2v_3\pm v_3v_4\pm v_1v_4$
with an even number of negative signs. Again, $f$ and $-f$ produce the same Euler
class. Since the number of maximal faces is even
the Euler classes for the two orientations on $\rho_\omega$
involve both either an even or an odd number of negative signs.
We have 16 choices for the Euler class, and eight of them admit exactly two non isomorphic complex structures and
eight of them do not admit a complex structure.

If $K=\partial\Delta[4]$, then we have 4 maximal faces, 4 vertices, 16 possible Euler classes
$\pm v_1v_2v_3\pm v_1v_2v_4\pm v_1v_3v_4\pm v_2v_3v_4$ and 16 non isomorphic bundles
$\xi_f$. A easy calculation shows that, in this case, the map $f\mapsto \omega_f$ is injective and surjective.
Hence, every $\omega\colon M_K\lra \{\pm 1\}$ is of the form $\omega_f$ and every
$\rho$ admits exactly two complex structures.

Now we assume that $K$ is a complex of dimension $n-1$ and
admits an $n$-coloring $g\colon K\lra \Delta[n]$. The composition $fg$ of a function
$f \colon [n]\lra \{\pm 1\}$ and the restriction $g|_{[m]}$ of $g$ to the set of vertices $[m]$
produces
the Euler class $e_\omega=\epsilon \sum_{\mu} v_\mu$
where $\epsilon\letbe \prod_{i\in [n]} f(i)$.
This shows that in this case, the bundle $\xi_\R$ admits at least  $2^n$ non-isomorphic
complex structures.
\end{exas}

\begin{rem} \label{cxstrucqtm}
For a simple polytope $P$ we denote by $K_P$ the dual of the boundary of $P$. Since $P$ is simple, $K_P$
is a simplicial
polytope, in particular a simplicial complex, homeomorphic to $S^{n-1}$. An orientation for $P$ or $K_P$
inherits an orientation to every maximal face $\mu \in K_P$, and hence, an orientation on the set $\mu$.
For a matrix $\Lambda\in \Z^{n\times m}$ and a maximal face $\mu\in K_P$
we denote by $\Lambda_\mu$ the maximal minor given by
the columns of $\Lambda$ associated to the vertices in $\mu$, but we order the columns according
to the orientation of $\mu$. That is the order of the columns is only fixed up to even permutations,
but such a permutation will not change the determinant.

 A dicharacteristic pair $(P,\Lambda)$ consists of an oriented simple polytope $P$
 of dimension $n$ with $m$ facets
and a matrix $\Lambda\in \Z^{n\times m}$ such that $\det A_\mu=\pm 1$ for all maximal faces of $K_P$.
Up to diffeomorphism, every oriented quasi toric manifold can be constructed from a
dicharacteristic pair $(P,\Lambda)$ (e.g. see \cite{bupa}). The Euler class of $(\tau_M)_{hT^n}$ is then given by
$e((\tau_M){hT^n})=\sum_{\mu\in M_{K_P}} \det \Lambda_\mu v_\mu$ \cite{dono}. In particular,
the real vector bundle $(\tau_M)_{hT^n}$ admits a complex structure if and only if their exists a function
$f\colon [m]\lra \{\pm 1\}$ and $\epsilon=\pm 1$ such that $\det A_\mu=\epsilon f(\mu)$ for
all maximal faces $\mu \in K_P$. This is exactly the condition sufficient and necessary for the existence of
an almost complex structure for $M^{2n}$ (e.g. see \cite[Corollary 5.54]{bupa} and the following remarks).
\end{rem}

If we stabilize the bundle $(\tau_M)_{hT^n}$, the picture will change.
Since $(\tau_M)_{hT^n}\oplus \R^2\cong \xi_\R\oplus \R^2$ (Theorem \ref{thm2}), Corollary \ref{cor7}
gives a complete list of the isomorphism classes of complex structures for $(\tau_M)_{hT^n}\oplus \R^{2s}$ for $s\geq 1$.

%
%
%
%
%
%
%
%
%


%
%
\end{document}